\documentclass[12pt]{amsart}
\usepackage{wrapfig}
\usepackage{graphicx}
\usepackage[all]{xy}
\usepackage{amsmath,amsthm,amssymb}
\usepackage{mathabx,epsfig}

\usepackage[margin=1in]{geometry}
\usepackage{amsmath,amscd}
\usepackage{mathtools}

\usepackage[all]{xy}

\usepackage{thmtools}
\usepackage{thm-restate}
\usepackage{hyperref}
\usepackage{cleveref}

\newtheorem{theorem}{Theorem}
\newtheorem{conjecture}{Conjecture}
\newtheorem{corollary}{Corollary}
\newtheorem{lemma}{Lemma}[theorem]

\usepackage[small,compact]{titlesec}
\usepackage{enumerate}
\usepackage{enumitem}
\usepackage{times}
\makeatletter

\newcommand{\Rmnum}[1]{\expandafter\@slowromancap\romannumeral #1@}
\makeatother

  \newcommand{\GL}{\operatorname{GL}}

\newcommand{\Mod}[1]{\ (\text{mod}\ #1)}
 \newcommand{\Aut}{\operatorname{Aut}}

\newcommand{\disc}{\operatorname{disc}}

\usepackage{xcolor}
\usepackage{titlesec}
\titleformat{\section}[block]{\color{black}\large\filcenter}{}{1em}{}
\titleformat{\subsection}[hang]{\bfseries}{}{1em}{}
\theoremstyle{remark}

\newtheorem{remark}{Remark}[theorem]
\begin{document}
\title{The Vojta conjecture implies Galois rigidity in dynamical families}
\author{Wade Hindes}
\date{\today}
\maketitle
\begin{abstract} We show that the Vojta (or Hall-Lang) conjecture implies that the arboreal Galois representations in a 1-parameter family of quadratic polynomials are surjective if and only if they surject onto some finite and uniform quotient. As an application, we use the Vojta conjecture, our uniformity theorem over $\mathbb{Q}(t)$, and Hilbert's irreducibility theorem to prove that the prime divisors of many quadratic orbits have density zero.   
\end{abstract}
\medskip 
To prove the surjectivity of the $\ell$-adic Galois representation attached to an elliptic curve, it suffices to prove the surjectivity onto some finite quotient. Namely, if $G\leq \GL_2(\mathbb{Z}_l)$ is a closed subgroup that surjects onto $\GL_2(\mathbb{Z}/\ell^n\mathbb{Z})$ for some small $n$, then $G$ must be equal to $\GL_2(\mathbb{Z}_l)$; see \cite{surjectivity}. However, in \cite{Jones1} Rafe Jones has suggested that such a rigidity is unlikely to hold if we replace $\GL_2(\mathbb{Z}_l)$ with the automorphism group $\Aut(T_\infty)$ of an infinite binary rooted tree $T_\infty$, replace the subgroup $G$ with the image of an arboreal Galois representation $G_\infty(\phi)$ attached to a quadratic polynomial $\phi\in\mathbb{Q}[x]$ (cf. \cite{B-J} and \cite{R.Jones}), and replace $\GL_2(\mathbb{Z}/\ell^n\mathbb{Z})$ with the automorphism group $\Aut(T_n)$ of a level $n$ binary rooted tree $T_n$; here the difficulty arises from the fact that the Frattini subgroup of $\Aut(T_\infty)$ has infinite index. To illustrate this point, Jones cleverly constructs a large quadratic polynomial 
\begin{equation}{\label{largeexample}}
\phi(x)=(x-88255775491812351975604)^2+88255775491812351975605,
\end{equation}  
satisfying $G_8(\phi)\cong \Aut(T_8)$ and $G_\infty(\phi)\not\cong\Aut(T_\infty)$; see \cite{Jones1}. However, our uniformity theorem over rational function fields (cf. \cite[Theorem 1]{uniformity}) together with Hilbert's irreducibility theorem suggest that such a rigidity holds in a 1-parameter family obtained by specializing a quadratic polynomial with polynomial coefficients. In this paper we prove such a claim, assuming some powerful (yet standard) conjectures in arithmetic geometry. As a consequence, we predict that the prime divisors of many quadratic orbits have density zero (cf. Corollary \ref{construction}).  

\emph{Notation}: We fix some notation. Let $\phi(x)=(x-\gamma(t))^2+c(t)$ for some polynomials $\gamma,c\in\mathbb{Z}[t]$ and let $\phi_a(x)=(x-\gamma(a))^2+c(a)$ be the specialization of $\phi$ at some integer $a$. Finally, we say that $\phi_a$ is \emph{stable} if every iterate of $\phi_a$ is irreducible. Then we have the following theorem regarding the arboreal Galois representations $G_\infty(\phi_a)$ in the family $\{\phi_a\}_{a\in\mathbb{Z}}$.

\begin{theorem}{\label{Rigidity}} Suppose that $\phi$ is not isotrivial, $\phi(\gamma)\cdot\phi^2(\gamma)\neq0$, and the Vojta (or Hall-Lang) conjecture holds. Then there exists an integer $n_\phi>0$ and an effectively computable finite set $F_\phi$ such that for all integers $a\notin F_\phi$, \[G_{n_\phi}(\phi_a)\cong\Aut(T_{n_\phi})\;\; \text{implies that}\;\;G_{\infty}(\phi_a)\cong\Aut(T_{\infty}).\] Furthermore, if $\phi_a$ is stable, then \[\sup_{\substack{a\not\in F_\phi\\ \phi_a\,\textup{is stable}}}\Big\{\big[\Aut(T_\infty): G_\infty(\phi_a)\big]\Big\}\;\;\text{is finite.}\]  
\end{theorem} 
The study of arboreal representations owes its beginnings to classical prime factorization problems in polynomial recurrences. Specifically, let $\phi\in\mathbb{Z}[x]$ be a polynomial with integer coefficients and let $b=b_0\in\mathbb{Z}$. For $n\geq1$, define the sequence $b_n=\phi(b_{n-1})=\phi^n(b_0)$. A fundamental object in dynamics, this set of numbers is called the \emph{orbit} of $b$ with respect to $\phi$ and is denoted 
\begin{equation}{\label{Orbit}}
\mathcal{O}_\phi(b):=\{b,\phi(b),\phi^2(b),\dots\}.
\end{equation} 
It is a classical question in number theory to ask whether $\mathcal{O}_\phi(b)$ contains infinitely many primes. At the moment, this question is well beyond reach. For example, if $\phi(x)=(x-1)^2+1$ and $b=3$, then $b_n=2^{2^n}+1$ are the Fermat numbers, which have been studied extensively \cite{Fermat}. 

However, one can ask a more tractable question, which has connections to arboreal representations. Namely, how big is the set of prime divisors in a particular orbit? As a partial answer, it is known that $G_\infty(\phi)\cong\Aut(T_\infty(\phi))$ implies that the set 
\begin{equation}{\label{Divisors}}
\mathcal{P}_\phi(b):=\{\text{primes}\; p\;\big\vert\; b_n=\phi^n(b)\equiv{0}\bmod{p}\;\text{for some}\; n\geq1\} 
\end{equation} 
of prime divisors of $\mathcal{O}_\phi(b)$ has density zero \cite[Theorem 4.1]{R.Jones}. Intuitively, this means that if the Galois groups of iterates of $\phi$ are as large as possible, then the prime divisors in any particular orbit cannot accumulate. To illustrate this point and Theorem \ref{Rigidity}, we use our uniformity theorem over $\mathbb{Q}(t)$ (cf. \cite[Theorem 1]{uniformity}) and Hilbert's irreducibility theorem to predict that the prime divisors of many quadratic orbits have density zero.          
\begin{corollary}{\label{construction}} Suppose that $\phi$ satisfies the following conditions: 
\begin{enumerate}[itemsep=1.25mm] 
\item[\textup{(a)}] $\phi$ is not isotrivial
\item[\textup{(b)}] $G_{m_\phi}(\phi)\cong\Aut(T_{m_\phi})$ for $m_\phi$ given by \begin{displaymath}
   m_\phi := \left\{
     \begin{array}{lr}
       17 & ,\; \deg(\gamma)\neq \deg(c)\\
       2\cdot\log_2\left(78\cdot \frac{\deg(\gamma)}{\deg(c-\gamma)}+9\right) &,\; \deg(\gamma)=\deg(c)
     \end{array}
   \right\}.
\end{displaymath}  
\end{enumerate}  
Then the Vojta (or Hall-Lang) conjecture implies the following statements: 
\begin{enumerate}[itemsep=1.5mm] 
\item[\textup{(1)}] There exists a thin set $E_\phi$ such that $G_{\infty}(\phi_a)\cong\Aut(T_{\infty})$ for all $a\notin E_\phi$.
\item[\textup{(2)}] The density $\delta(\mathcal{P}_{\phi_a}(b))=0$ for all $b\in\mathbb{Z}$ and all $a\notin E_\phi$.   
\end{enumerate} 
\end{corollary}

Before we begin the proof of the theorem, we remind the reader of the relevant conjectural height bounds in arithmetic geometry. In keeping with standard notation, we let $h:\widebar{\mathbb{Q}}\rightarrow\mathbb{R}_{\geq0}$ be the (absolute) logrithmic height of an algebraic number \cite[VIII.5]{SilvElliptic}. Similarly, for $f\in\widebar{\mathbb{Q}}[x]$, we let $h(f)$ be the maximum of the heights of the coefficients of $f$.   
\begin{conjecture}{\label{conjecture}} For all $d\geq3$ there exist constants $C_1=C_1(d)$ and $C_2=C_2(d)$ so that for all $f\in\mathbb{Z}[x]$ of degree $d$ with $\disc(f)\neq0$, if $x,y\in\mathbb{Z}$ satisfy $y^2=f(x)$, then 
\begin{equation}{\label{b}} h(x)\leq C_1\cdot h(f)+C_2.
\end{equation}  
\end{conjecture} 
\begin{remark} Versions of Conjecture \ref{conjecture} were made by Hall and Lang (cf. \cite[IV.7]{SilvElliptic}) for $d\leq4$. On the other hand, for larger $d$ Conjecture \ref{conjecture} is a consequence of the Vojta conjecture; see the main result of \cite{Heightuniformity} or \cite[Example 5]{RationalPoints}. 
\end{remark}  
\begin{proof}[Proof \,(Theorem)] To prove Theorem \ref{Rigidity}, we need only assume that the bounds in Conjecture \ref{conjecture} hold for a single value of $d\geq3$. Therefore, without loss of generality, we assume that (\ref{b}) holds for $d=3$. Throughout the proof, we use the following fact: for all $f\in\mathbb{Q}[x]$ of degree $d$ there are constants $\mathcal{C}_{1,f}$ and $\mathcal{C}_{2,f}$, depending on $f$, such that 
\begin{equation}{\label{morphism}} d\cdot h(\alpha)-\mathcal{C}_{1,f}\leq h(\alpha)\leq d\cdot h(\alpha)+\mathcal{C}_{2,f}\;\;\; \text{for all}\; \alpha\in\bar{\mathbb{Q}}; 
\end{equation}  
see \cite[Theorem 3.11]{Silv-Dyn}. We fix some notation. Define the affine transformation $\lambda_a(x)=x+\gamma(a)$, and let $\sigma_a$ be the quadratic polynomial given by conjugating $\phi_a$ by $\lambda_a$, that is \[\sigma_a(x):=\lambda_a^{-1}\circ\phi_a\circ\lambda_a(x)=x^2+c(a)-\gamma(a).\]
The triangle inequality on the absolute values of $\mathbb{Q}$ and (\ref{morphism}) imply that 
\begin{equation}{\label{upperbound}}
h(\lambda_a^{-1}(\alpha))\leq h(\alpha)+2\cdot\log(3)+\deg(\gamma)\cdot h(a)+h(\gamma) \;\;\text{for all}\; \alpha\in\bar{\mathbb{Q}}.  
\end{equation}
On the other hand, there is a lower bound 
\begin{equation}{\label{lowerbound}}
 h(\alpha)-\deg(\gamma)\cdot h(a)-A_{1,\phi}\leq h(\lambda_a^{-1}(\alpha))\;\;\text{for all}\; \alpha\in\bar{\mathbb{Q}}
\end{equation}
and some positive constant $A_{1,\phi}$; see \cite[Theorem 3.11]{Silv-Dyn}. Moreover, repeated application of the triangle inequality implies that 
\begin{equation}{\label{orbitbound}}
h(\sigma_a^m(\alpha))\leq 2^m\cdot\big(h(\alpha)+A_{2,\phi}+\deg(c-\gamma)\cdot h(a)\big) \;\;\text{for all}\; \alpha\in\bar{\mathbb{Q}} 
\end{equation}
and some positive constant $A_{2,\phi}$. Finally, since $\phi$ is not isotrivial, $\deg(c-\gamma)\neq0$. In particular, (\ref{morphism}) implies that there exists a computable, positive constant $B_{1,\phi}$ such that
\begin{equation}{\label{isotrivial}}  
\deg(c-\gamma)\cdot h(a)-B_{1,\phi}\leq h\big(c(a)-\gamma(a)\big), \;\;\; \text{for all}\; a\in\bar{\mathbb{Q}}.
\end{equation} 
From here, we derive Theorem \ref{Rigidity} from the following lemma, which ensures the existence of so called square-free, primitive prime divisors in the critical orbit after a uniform number of iterates; compare to \cite{Tucker} and \cite{Silv-Vojta}.     
\begin{lemma}{\label{primdiv}} Assume the Vojta (or Hall-Lang) conjecture and suppose that $a\in\mathbb{Z}$ satisfies the following properties:
\begin{enumerate}[itemsep=2mm] 
\item[\textup{(1)}]\; $\phi_a(\gamma(a))\cdot\phi_a^2(\gamma(a))\neq0$,
\item[\textup{(2)}]\; $c(a)-\gamma(a)\notin\{-2,-1,0,\}$,
\item[\textup{(3)}]\; $\deg(c-\gamma)\cdot h(a)-B_{1,\phi}>0.$
\end{enumerate}
Then there is an $n_\phi>0$ (not depending on $a$) such that for all $n>n_\phi$ there exists an odd prime $p_n$ with the following property: 
\begin{equation}{\label{PrimDiv}}
v_{p_n}(\phi_a^n(\gamma(a)))\notequiv0\Mod{2}\;\;\;\text{and}\;\;\; v_{p_n}(\phi_a^j(\gamma(a)))=0\;\;\text{for all}\;\;1\leq j\leq n-1; 
\end{equation} 
here $v_{p}$ denotes the normalized p-adic valuation. Such a prime $p_n$ is called a square-free, primitive prime divisor for $\phi_a^n(\gamma(a))$.   
\end{lemma} 
\begin{proof}  
For every $n$, write $\phi_a^n(\gamma(a))=2^{e_n}\cdot d_n\cdot y_n^2$ for $e_n\in\{0,1\}$ and some odd, square-free integer $d_n$. Our goal is to first prove that the $d_n$ grow rapidly, and from there deduce that eventually each new $d_n$ is divisible by a new prime. Note that if the $d_n$'s were to grow slowly, then (ignoring the power of 2 for now), there would be values of $d$ for which the curve $dY^2 = \phi_a^2(X)$ has a rational point with very large coordinates compared to the height of its defining equation. Quantifying this idea leads to a contradiction of Conjecture \ref{conjecture}. 

If $n$ is such that no prime $p_n$ as in (\ref{PrimDiv}) exists, then $d_n$ is a unit or $d_n=\prod p_i$ for some primes $p_i\in \mathbb{Z}$ satisfying $p_i\big\vert\phi_a^{m_i}(\gamma(a))$ and $1\leq m_i\leq n-1$. On the other hand, if $d_n$ is not a unit, then each  $p_i\big\vert\phi_a^{n-m_i}(0)$, since $p_i\big\vert\phi_a^{m_i}(\gamma(a))$ and $p_i\big\vert\phi_a^{n}(\gamma(a))$. To see this, note that \[\phi_a^{n-m_i}(0)\equiv\phi_a^{n-m_i}(\phi_a^{m_i}(\gamma(a)))\equiv\phi_a^n(\gamma(a))\equiv0\Mod{p_i}.\] In particular, when $d_n$ is not a unit, we have the refinement:    
 \begin{equation}{\label{refinement}}  d_n=\prod p_i,\;\;\text{where}\;\; p_i\big\vert\phi_a^{t_i}(\gamma(a))\;\text{or}\; p_i\big\vert\phi_a^{t_i}(0)\;\; \text{for some}\; 1\leq t_i\leq\Big\lfloor \frac{n}{2}\Big\rfloor.   
\end{equation}  

Our goal is to show that $n$ is bounded independently of $a$. To do this, define the elliptic curve 
\begin{equation}{\label{curve}} C_{\phi_a}^{(d_n)}: Y^2=2^{e_n}\cdot d_n\cdot(X-c(a))\cdot \phi_a(X). 
\end{equation}
Note that $C_{\phi_a}^{(d_n)}$ is nonsingular by assumption (1) of Lemma \ref{primdiv}. Then we have the integral point 
\begin{equation}{\label{point}}
\Big(\;\phi_a^{n-1}(\gamma(a))\,,\; 2^{e_n}\cdot d_n\cdot y_n\cdot\big(\phi_a^{n-2}(\gamma(a))-\gamma(a)\big) \Big)\in C_{\phi_a}^{(d_n)}(\mathbb{Z}).  
\end{equation} 
In particular, the Hall-Lang conjecture (cf. Conjecture \ref{conjecture} for $d=3$) on integral points of elliptic curves implies that
\begin{equation}{\label{bound1}}
h\big(\phi_a^{n-1}(\gamma(a))\big)\leq \kappa_1\cdot h(d_n)+\kappa_2\cdot h(a)+\kappa_3
\end{equation} 
for some absolute constants $\kappa_i>0$.  

Assuming the Vojta conjecture, one obtains a similar bound as in (\ref{bound1}); we simply replace $C_{\phi_a}^{(d_n)}$ and the point on (\ref{point}) with the genus two curve $Y^2=2^{e_n}\cdot d_n\cdot(X-c(a))\cdot\phi^2(X)$ and its corresponding point with $X$-coordinate $\phi_a^{n-2}(\gamma(a))$; see \cite{Me} for more on these hyperelliptic curves defined by iteration. Finally, apply the main result of \cite{Heightuniformity}; see also \cite[Example 5]{RationalPoints}.    

Combining (\ref{upperbound}) and (\ref{bound1}), we obtain
\begin{equation}{\label{combin}}
h(\lambda_a^{-1}(\phi_a^{n-1}(\gamma(a))))\leq \kappa_1\cdot h(d_n)+\kappa_{2,\phi}\cdot h(a)+\kappa_{3,\phi}, 
\end{equation} 
where the constants $\kappa_{2,\phi}$ and $\kappa_{3,\phi}$ depend of $\phi$. However, \[\lambda_a^{-1}(\phi_a^{n-1}(\gamma(a)))=\lambda_a^{-1}\circ\phi_a^{n-1}\circ\lambda_a(0)=\sigma_a^{n-1}(0).\]
On the other hand, the canonical height (cf. \cite[3.4]{Silv-Dyn}) satisfies 
\[\big|\hat{h}_{\sigma_a}(x)-h(x)\big|\leq h\big(c(a)-\gamma(a)\big)+\log(2)\leq\deg(c-\gamma)\cdot h(a)+B_{1,\phi}+\log(2),\] for all $x\in\mathbb{Q}$; see \cite[Lemma 12]{ingram}. In particular, we apply this bound to $x=\sigma_a^{n-1}(0)$ and use (\ref{combin}) to conclude that  
\begin{equation}{\label{bound2}} 
2^{n-1}\cdot \hat{h}_{\sigma_a}(0)=\hat{h}_{\sigma_a}(\sigma_a^{n-1}(0))\leq \kappa_1\cdot h(d_n)+\kappa_{2,\phi}'\cdot h(a)+\kappa_{3,\phi}'. 
\end{equation}
Moreover, Ingram has shown in \cite[Proposition 11]{ingram}, that 
\[\hat{h}_{\sigma_a}(x)\geq \frac{1}{32}\max\big\{h\big(c(a)-\gamma(a)\big),1\big\}, \;\;\text{for all wandering points}\;x\in\mathbb{Q}.\] 
By assumption $c(a)-\gamma(a)\notin\{-2,-1,0\}$, so that $0$ is a wandering point of $\sigma_a$ (equivalently, $\phi_a$ is not postcritically finite). To see this, note that if $0$ is not wandering, then $c(a)-\gamma(a)$  belongs to the Mandelbrot set $\mathcal{M}$ over the complex numbers; see \cite[\S4.24]{Silv-Dyn}. In particular, \cite[Proposition 4.19]{Silv-Dyn} implies that $|c(a)-\gamma(a)|\leq2$, where $|\cdot|$ denotes the complex absolute value. Hence the absolute logarithmic height of $c(a)-\gamma(a)$ is at most $\log(2)$. One checks that this implies that $c(a)-\gamma(a)\in\{0,-1,-2\}$ as claimed. 

On the other hand, we know there exists a positive constant $B_{1,\phi}$ as on (\ref{isotrivial}). Consolidating this fact with the lower bound on $\hat{h}_{\sigma_a}(0)$ and the bound on (\ref{bound2}), we obtain that   
\begin{equation}{\label{bound3}} 2^{n-6}\cdot\big(\deg(c-\gamma)\cdot h(a)-B_{1,\phi}\big)\leq \kappa_1\cdot h(d_n)+\kappa_{2,\phi}'\cdot h(a)+\kappa_{3,\phi}'.
\end{equation}
The left hand side of (\ref{bound3}) is of our desired shape. It remains to bound $h(d_n)$ in terms of $h(a)$, to complete the proof of Lemma \ref{primdiv}. To do this, note that (\ref{lowerbound}) and (\ref{refinement}) together imply that 
\begin{equation}
\begin{split}
h(d_n)&\leq\sum_{i=1}^{\lfloor \frac{n}{2}\rfloor}h(\phi_a^i(\gamma(a)))+\sum_{j=1}^{\lfloor \frac{n}{2}\rfloor}h(\phi_a^j(0)) \\
 & \leq \sum_{i=1}^{\lfloor \frac{n}{2}\rfloor}h(\sigma_a^i(0))+\sum_{j=1}^{\lfloor \frac{n}{2}\rfloor}h(\sigma_a^j(\gamma(a)))+n\cdot\big(\deg(\gamma)\cdot h(a)+A_{1,\phi} \big) 
\end{split}
\end{equation}
On the other hand, (\ref{orbitbound}) implies that \[h(\sigma_a^i(0))\leq 2^i\cdot\big(A_{2,\phi}+\deg(c-\gamma)\cdot h(a)\big)\;\;\text{and}\;\; h(\sigma_a^j(\gamma(a)))\leq2^j\cdot\big(h(\gamma(a))+A_{2,\phi}+\deg(c-\gamma)\cdot h(a)\big).\] However, there exists a positive constant $A_{3,\phi}$, such that $h(\gamma(a))\leq \deg(\gamma)\cdot h(a)+A_{3,\phi}$. Hence, 

\begin{equation}{\label{bound4}} 
\begin{split} 
h(d_n)&\leq 2^{\lfloor\frac{n}{2}\rfloor+2}\cdot\big(A_{2,\phi}+\deg(c-\gamma)\cdot h(a)\big)+2^{\lfloor\frac{n}{2}\rfloor+1}\cdot\big(\deg(\gamma)\cdot h(a)+A_{3,\phi}\big) \\
 &\;\;\;\; + n\cdot\big(\deg(\gamma)\cdot h(a)+A_{1,\phi} \big)
\end{split}   
\end{equation}
Combining (\ref{bound3}) and (\ref{bound4}), we see that
\begin{equation}{\label{bound5}} 
\begin{split}
2^{n-6}&\leq\bigg(\frac{\kappa_1\cdot\deg(c-\gamma)\cdot h(a)+\kappa_1\cdot A_{2,\phi}}{\deg(c-\gamma)\cdot h(a)-B_{1,\phi}}\bigg)\cdot 2^{\lfloor\frac{n}{2}\rfloor+2}+ \bigg(\frac{\kappa_1\cdot\deg(\gamma)\cdot h(a)+\kappa_1\cdot A_{3,\phi}}{\deg(c-\gamma)\cdot h(a)-B_{1,\phi}}\bigg)\cdot 2^{\lfloor\frac{n}{2}\rfloor+1} \\
\\
& +\bigg(\frac{\kappa_1\cdot\deg(\gamma)\cdot h(a)+\kappa_1\cdot A_{1,\phi}}{\deg(c-\gamma)\cdot h(a)-B_{1,\phi}}\bigg)\cdot n+\bigg(\frac{\kappa_{2,\phi}'\cdot h(a)+\kappa_{3,\phi}'}{\deg(c-\gamma)\cdot h(a)-B_{1,\phi}}\bigg). 
\end{split}  
\end{equation} 
\\
However, as a real valued function, any linear fractional transformations $\rho(x)=\frac{ax+b}{cx+d}$ is bounded away from its poles. Hence, if we view $h(a)$ as the variable $x$, we see that 
\begin{equation} 2^{n-6}\leq M_{1,\phi}\cdot 2^{\lfloor\frac{n}{2}\rfloor+2}+ M_{2,\phi}\cdot 2^{\lfloor\frac{n}{2}\rfloor+1}+ M_{3,\phi}\cdot n+M_{4,\phi},  
\end{equation}
since $\deg(c-\gamma)\cdot h(a)-B_{1,\phi}>0$. In particular, if $M_\phi:=\max\{M_{i,\phi}\}$, then 
\begin{equation}{\label{finalbound}}
n\leq\max\{6,2\cdot\log_2(M_\phi)+19\}.
\end{equation} 
 Therefore, we have bounded $n$ independently of $a$. This completes the proof of Lemma \ref{primdiv}.       
\end{proof}   
With the lemma in place, we return to the proof of Theorem \ref{Rigidity}. Let $P_\phi(t)$ be the polynomial whose roots cut out conditions (1) and (2) of Lemma \ref{primdiv}, that is 
\begin{equation}{\label{polynomial}}
P_\phi(t):=\phi(\gamma(t))\cdot\phi^2(\gamma(t))\cdot\big(c(t)-\gamma(t)\big)\cdot \big(c(t)-\gamma(t)+1\big)\cdot \big(c(t)-\gamma(t)+2\big).
\end{equation} 
From here, we can define the finite set of exceptional specializations:  
\begin{equation}{\label{set}}
F_\phi:=\bigg\{a\in\mathbb{Z}\;\bigg\vert\;  P_\phi(a)=0\;\;\;\text{or}\;\;\; h(a)\leq \frac{B_{1,\phi}}{\deg(c-\gamma)}\bigg\}.
\end{equation}
\begin{remark} The set $F_\phi$ can be explicitly computed provided that the Nullstellensatz step in computing $B_{1,\phi}$ can be carried effectively; see the proof of the lower bound in \cite[Proposition 3.11]{Silv-Dyn}.  
\end{remark}    
We now set $n_\phi:=1+\max\{2,\;2\cdot\log_2(M_\phi)+19\}$ and suppose that $G_{n_\phi}(\phi_a)\cong\Aut(T_{n_\phi})$ and that $a\notin F_\phi$. To prove that $G_{\infty}(\phi_a)\cong\Aut(T_{\infty})$, we first show that $\phi_a$ is stable. 

If $\phi_a$ is not stable, then \cite[Proposition 4.2]{Jones2} implies that $\phi_a^n(\gamma(a))$ is a square for some $n\geq1$. However, by Lemma \ref{primdiv}, such an $n$ must be less than $n_\phi$. In particular, $G_{n-1}(\phi_a)\cong\Aut(T_{n-1})$, as the Galois group of the larger iterate $\phi_a^{n_\phi}$ is maximal. Moreover, since the full automorphism group of the preimage tree acts transitively on the roots of $\phi_a^{n-1}$, we conclude that $\phi_a^{n-1}$ must be irreducible. 

On the other hand, \cite[Lemma 1]{uniformity} implies that $K_n(\phi_a)/K_{n-1}(\phi_a)$ is not maximal, since $\phi_a^n(\gamma(a))$ is a square in $K_{n-1}(\phi_a)$; in fact, it is already a square over the rational numbers. Therefore, $G_n(\phi_a)\ncong\Aut(T_n)$, contradicting our assumption that $G_{n_\phi}(\phi_a)\cong\Aut(T_{n_\phi})$. We conclude that $\phi_a$ is stable.

Now, let $m$ be any integer and suppose that the subextension $K_{m}(\phi_a)/K_{m-1}(\phi_a)$ is not maximal. In particular, $m>n_{\phi}$ since $G_{n_\phi}(\phi_a)\cong\Aut(T_{n_\phi})$. However, we have shown that $\phi_a$ is stable, hence $\phi_a^m(\gamma(a))\in \big(K_{m-1}(\phi_a)\big)^2$; see \cite[Lemma 1]{uniformity}. Hence, if we write $\phi_a^m(\gamma(a))=2^{e_m}\cdot d_m\cdot y_m^2$ for some $y_m,d_m\in \mathbb{Z}$ such that $d_m$ is a unit or an odd square-free integer, then the primes dividing $d_m$ must ramify in $K_{m-1}(\phi_a)$.

On the other hand, by \cite[Corollary 2, p.159]{Number theory}, we see that the primes which ramify in $K_{m-1}$ must divide the discriminant of $\phi^{m-1}$. Let $\Delta_n$ be the discriminant of $\phi^n$. Then we have the following formula, given in \cite[Lemma 2.6]{Jones2}: 
 \begin{equation}{\label{discriminant}}\Delta_n=\pm\Delta_{n-1}^2\cdot 2^{2^n}\cdot \phi^n(\gamma).
 \end{equation} 
 In particular, if $d_m$ is not a unit, then $d_m=\prod p_i$ for some odd primes $p_i\in \mathbb{Z}$ such that $p_i\big\vert\phi^{m_i}(\gamma)$ and $1\leq m_i\leq m-1$. However, since $m>n_\phi$, Lemma \ref{primdiv} implies that there exists a prime divisor of $d_m$ which is coprime to all lower iterates, a contradiction. Hence, $K_m(\phi_a)/K_{m-1}(\phi_a)$ is maximal for all $m$. It follows that $G_{\infty}(\phi_a)\cong\Aut(T_{\infty})$, completing the first statement of the Theorem \ref{Rigidity}.

Similarly, if $a\notin F_\phi$ and $\phi_a$ is stable, then Lemma \ref{primdiv} implies that the subextensions $K_m(\phi_a)/K_{m-1}(\phi_a)$ are maximal for all $m>n_\phi$. Hence $G_\infty(\phi_a)$ is a finite index subgroup of $\Aut(T_\infty)$, and
\[\big[\Aut(T_\infty): G_\infty(\phi_a)\big]=\big[\Aut(T_{n_\phi}): G_{n_\phi}(\phi_a)\big]= \frac{2^{2^{n_\phi}-1}}{\big[K_{n_\phi}(\phi_a):\mathbb{Q}\big]}\leq \frac{2^{2^{n_\phi}-1}}{2^{n_\phi}}=2^{2^{n_\phi}-n_{\phi}-1},\]
since $\phi_a^{n_\phi}$ is an irreducible polynomial of degree $2^{n_\phi}$ over the rational numbers. In particular, the index bound does not depend on $a$, which completes the proof of the theorem.                         
\end{proof} 
If we fix $\phi$, it is natural to ask to what extent the corresponding $n_\phi$ is computable (at least conjecturally). However, since the constants appearing in Conjecture \ref{conjecture} are not explicit, we cannot make the proof of Theorem \ref{Rigidity} effective. Nonetheless, it is possible to use additional techniques in the theory of rational points on curves to classify the Galois behavior of small iterates and produce a conjectural $n_\phi$. For instance, we make the following conjecture when $\phi(x)=x^2+t$.
\begin{conjecture}{\label{conj:rigidity}} If $\phi_a(x)=x^2+a$, then for all $a\in\mathbb{Z}$,  
\[G_3(\phi_a)\cong\Aut(T_3)\;\;\;\text{implies that}\;\;\;G_\infty(\phi_a)\cong\Aut(T_\infty).\]
In particular, if $a\neq3$, then $G_2(\phi_a)\cong\Aut(T_2)$ implies that $G_\infty(\phi_c)\cong\Aut(T_\infty)$.  
\end{conjecture} 
\begin{remark} The evidence for Conjecture \ref{conj:rigidity} comes from \cite[Theorem 1.1]{Me}. There we prove that $G_3(\phi_a)\cong\Aut(T_3)$ implies $G_4(\phi_a)\cong\Aut(T_4)$. Furthermore, we show that if $a\neq3$ and $G_2(\phi_a)\cong\Aut(T_2)$, then $G_4(\phi_a)\cong\Aut(T_4)$. 
\end{remark} 
We close with the proof of Corollary \ref{construction}. 
\begin{proof}[Proof \,(Corollary)] If $\phi$ is not isotrivial and $G_{m_\phi}(\phi)\cong\Aut(T_{m_\phi})$, then it follows from \cite[Theorem 1]{uniformity} that $G_{\infty}(\phi)\cong\Aut(T_{\infty})$. In particular, $\phi(\gamma)\cdot\phi^2(\gamma)\neq 0$ and there exists a finite set $F_\phi$ and an integer $n_\phi$ such that 
\begin{equation}{\label{implies}}
G_{n_\phi}(\phi_a)\cong\Aut(T_{n_\phi})\;\; \text{implies that}\;\;G_{\infty}(\phi_a)\cong\Aut(T_{\infty}),
\end{equation} 
for all $a\notin F_\phi$; see Theorem \ref{Rigidity} above. On the other hand, since $G_{\infty}(\phi)\cong\Aut(T_{\infty})$, we conclude that $G_{n_\phi}(\phi)\cong\Aut(T_{n_\phi})$ and Hilbert's irreducibility theorem implies that the set 
\begin{equation}{\label{Z}}
Z_\phi:=\big\{a\in\mathbb{Z}\;\big\vert\; G_{n_\phi}(\phi_a)\not\cong\Aut(T_{n_\phi}) \big\}
\end{equation} 
is thin; see \cite[Theorem 3.4.1]{Serre}. Hence, the set $E_\phi:= Z_\phi\cap F_\phi$ is also thin, and if $a\not\in E_\phi$, then (\ref{implies}) and (\ref{Z}) imply that $G_{\infty}(\phi_a)\cong\Aut(T_{\infty})$. In particular, the density $\delta(\mathcal{P}_{\phi_a}(b))$ of prime divisors of $\mathcal{O}_{\phi_a}(b)$ is zero; see \cite[Theorem 4.1]{R.Jones}. This completes the proof of Corollary \ref{construction}.               
\end{proof} 

\textbf{Acknowledgements:} It is a pleasure to thank Joe Silverman for the many discussions related to the work in this paper. I also thank Michael Stoll for relaying to me the question of Odoni regarding the existence of the thin set $E_\phi$ for the family $\phi_a(x)=x^2+a$ (cf. Corollary \ref{construction}).

\end{document}